\documentclass[12pt]{amsart}

\usepackage{amsmath,amssymb,amscd,amsthm}
\usepackage{hyperref}
\usepackage[left=3cm,right=3cm,top=3cm,bottom=3cm,includeheadfoot]{geometry}

\makeatletter
\@namedef{subjclassname@2010}{%
  \textup{2010} Mathematics Subject Classification}
\makeatother

\newtheorem{theorem}{Theorem}

\newtheorem {corollary}{Corollary}
\newtheorem {lemma}{Lemma}

\theoremstyle{definition}

\newtheorem{example}{Example}

\DeclareMathOperator\Sym{Sym}

\DeclareMathOperator\ctop{top}
\DeclareMathOperator\rank{rank}

\begin{document}

\title[]{Numerical invariants of Fano schemes of linear subspaces on complete intersections}

\author[]{Dang Tuan Hiep}

\address{National Center for Theoretical Sciences\\ No. 1 Sec. 4 Roosevelt Rd., National Taiwan University\\ Taipei, 106, Taiwan}
\email{hdang@ncts.ntu.edu.tw}

\begin{abstract}
The goal of this paper is to explore the genus and degree of the Fano scheme of linear subspaces on a complete intersection in a complex projective space. Firstly, suppose that the expected dimension of the Fano scheme is one, we prove a genus-degree formula. Secondly, we give a degree formula for the Fano scheme.
\end{abstract}

\subjclass[2010]{14C05, 14C17}

\keywords{Fano scheme, genus-degree formula, degree formula}

\date{\today}

\maketitle

\section{Introduction}

Let $X$ be a general complete intersection of type $\underline{d} = (d_1,\ldots,d_r)$ in the projective space $\mathbb P^n$ over the complex field $\mathbb C$, provided that $n,d_1,\ldots,d_r$ are natural numbers with $n\geq 4,d_i\geq 2$ for all $i$. Recall that the Fano scheme $F_k(X)$ parametrizing linear subspaces of dimension $k$ contained in $X$ is a smooth subscheme of the Grassmannian $G(k+1,n+1)$ of linear subspaces of dimension $k$ in $\mathbb P^{n}$, provided that 
$$(k+1)(n-k) \geq \sum_{i=1}^r\binom{d_i+k}{k}$$ 
and $X$ is not a quadric, in which last case $n \geq 2k+ r$ is required (see \cite[Corollary 2.2]{B} or \cite[Theorem 2.1]{DM}). For the basic properties of $F_k(X)$, we refer to \cite{AK, BV, B, DM, L}. For the important applications of $F_k(X)$ to the geometry of $X$, we refer to \cite{BM, CG, ELV}. Some numerical invariants of $F_k(X)$ have been studied by many authors. For instance, the Picard number of $F_k(X)$ was shown in \cite[Proposition 3.1]{DM} and \cite[Theorem 0.3]{J}. The degree of $F_k(X)$ under the Pl\"ucker embedding were formulated in \cite[Theorem 4.3]{DM} and \cite[Theorem 1.1]{H}. In this paper, we explore the numerical invariants of $F_k(X)$. For convenience, we set 
$$\delta(n,\underline{d},k) = (k+1)(n-k) - \sum_{i=1}^r\binom{d_i+k}{k},$$
which is the expected dimension of $F_k(X)$. Our main results are the following: 

\begin{theorem}\label{genusdegree}
With the notations as above, if $\delta(n,\underline{d},k) = 1$, then the Fano scheme $F_k(X)$ is a connected smooth projective curve of degree $d$ and genus $g$ satisfying the following formula:
$$g = 1 + \frac{\displaystyle\sum_{i=1}^r\binom{d_i+k}{k+1} - n - 1}{2}d.$$
\end{theorem}

\begin{theorem}\label{degre}
With the notations as above, if $\delta(n,\underline{d},k) \geq 0$, then the degree $d$ of $F_k(X)$ under the Pl\"ucker embedding is given by 
$$d = \frac{c(n,\underline{d},k)}{(k+1)!},$$
where $c(n,\underline{d},k)$ is the coefficient of $x_0^n\cdots x_k^n$ in the polynomial
$$\prod_{i=1}^r\prod_{a_0+\cdots+a_k=d_i,a_i\in\mathbb N}(a_0x_0+\cdots+a_kx_k)(x_0+\cdots+x_k)^{\delta(n,\underline{d},k)}\prod_{i\neq j}(x_i-x_j).$$
\end{theorem}

The statement of Theorem \ref{genusdegree} can be viewed as a natural generalization of the work of Markusevich \cite{M} and Tennision \cite{T}. Meanwhile, the statement of Theorem \ref{degre} seems to be similar with that of \cite[Theorem 4.3]{DM}. However, here we consider the coefficient of the monomial $x_0^n\cdots x_k^n$ in the product of the polynomial 
$$\prod_{i=1}^r\prod_{a_0+\cdots+a_k=d_i,a_i\in\mathbb N}(a_0x_0+\cdots+a_kx_k)(x_0+\cdots+x_k)^{\delta(n,\underline{d},k)}$$
by the discriminant 
$$\Delta = \prod_{i\neq j}(x_i-x_j)$$
instead of that of the monomial $x_0^n x_1^{n-1}\cdots x_k^{n-k}$ in the product of the same polynomial by the Vandermonde determinant  
$$a_{\delta} = \prod_{i<j}(x_i-x_j).$$
In particular, the proof of Theorem \ref{degre} is completely different from that of \cite[Theorem 4.3]{DM}. We apply an integral formula for Grassmannians which has been recently explored in \cite[Theorem 3]{H1}. The rest of the paper is organized as follows: Section 2 present the proof of Theorem \ref{genusdegree}. Theorem \ref{degre} is proved in Section 3.

\section{Proof of Theorem \ref{genusdegree}}

For the proof of Theorem \ref{genusdegree}, we first prove some lemmas.

\begin{lemma}\label{lem1}
If $E$ is a vector bundle of rank $k+1$ and $\Sym^mE$ is the $m$-th symmetric power of $E$, then the rank of $\Sym^mE$ is $\binom{m+k}{k}$ and
$$c_1(\Sym^mE) = \binom{m+k}{k+1}c_1(E),$$
where $c_1(E)$ is the first Chern class of the vector bundle $E$.
\end{lemma}
\begin{proof}
We prove by induction on $m$ and $k$. For $m=1$, we have $\Sym^1E = E$, so the rank of $\Sym^1E$ is $k+1$ and $c_1(\Sym^1E) = c_1(E)$ for all $k$, the conclusion is true. Assume that the conclusion holds for all $m'\leq m$ and $k'\leq k$. By arguments similar to those in \cite[Lemma 7.6]{EH}, we consider an exact sequence of vector bundles
$$0 \longrightarrow L \longrightarrow E \longrightarrow E' \longrightarrow 0,$$
where $L$ is a line bundle and $E'$ is a vector bundle of rank $k$. Thus we have
$$c_1(E) = c_1(L) + c_1(E').$$
The universal property of the symmetric powers (see, for example, \cite[Proposition A.2.2.d]{E}) shows that for each $m \geq 1$ there is an exact sequence
$$0 \longrightarrow L \otimes \Sym^{m-1}E \longrightarrow \Sym^mE \longrightarrow \Sym^mE' \longrightarrow 0.$$
By induction, the rank of $\Sym^mE$ is 
$$\binom{m+k-1}{k-1} + \binom{m-1+k}{k} = \binom{m+k}{k},$$
and
\begin{eqnarray*}
c_1(\Sym^mE)& = & c_1(L \otimes \Sym^{m-1}E) + c_1(\Sym^mE')\\
			& = & \binom{m+k-1}{k}c_1(L) + c_1(\Sym^{m-1}E) + \binom{m+k-1}{k}c_1(E')\\
			& = & \binom{m+k-1}{k}c_1(E) + \binom{m+k-1}{k+1}c_1(E)\\
			& = & \binom{m+k}{k+1}c_1(E).
\end{eqnarray*}
\end{proof}

\begin{lemma}(\cite[Proposition 5.25]{EH}) \label{lem2}
The first Chern class of the tangent bundle $T_G$ of the Grassmannian $G = G(k+1,n+1)$ is
$$c_1(T_G) = (n+1)\sigma_1,$$
where $\sigma_1 = c_1(S^\vee) = c_1(Q)$ is a hyperplane section of $G$ in the  Pl\"ucker embedding, $S$ and $Q$ are respectively universal sub and quotient bundles on $G$, and $S^\vee$ is the dual of $S$.
\end{lemma}

\begin{proof}
The tangent bundle $T_G$ is expressed as the tensor product $S^\vee \otimes Q$. Thus we have
$$c_1(T_G) = c_1(S^\vee \otimes Q) = \rank(Q)c_1(S^\vee) + \rank(S^\vee)c_1(Q).$$ 
Since $\rank(Q) = n-k, \rank(S^\vee) = k+1$, and $c_1(S^\vee) = c_1(Q) = \sigma_1$, hence $c_1(T_G) = (n+1)\sigma_1$ as desired.
\end{proof}

\begin{lemma}(\cite[Proposition 6.4]{EH})\label{lem3}
Let $X\subset\mathbb P^n$ be a general complete intersection of type $(d_1,\ldots,d_r)$. The scheme $F = F_k(X)$ is the zero locus of a global section of the vector bundle
$$\mathcal F = \bigoplus_{i=1}^r \Sym^{d_i}S^\vee.$$
\end{lemma}

\begin{proof}
Assume that $X$ is the intersection of $r$ hypersurfaces $X_1, \ldots, X_r$ with $\deg(X_i) = d_i$ for all $i$. Each $F_k(X_i)$ is the zero locus of a global section $s_i$ of $\Sym^{d_i}S^\vee$. Thus the scheme $F$, which is the intersection of the $F_k(X_i)$, is the zero locus of a global section $s = (s_1,\ldots,s_r)$ of the vector bundle $\mathcal F$. 
\end{proof}

\begin{proof}[Proof of Theorem \ref{genusdegree}]
The dimension and smoothness of the Fano scheme $F = F_k(X)$ are proved in \cite[Corollary 2.2]{B} and \cite[Th\'eor\`eme 2.1]{DM}, also see further discussion in \cite[Chapter 6]{EH} when $X$ is a hypersurface. The connectedness of $F$ is proved in \cite[Theorem 4.1]{B}, also see in \cite[Theorem 0.1]{L} for the hypersurface case. 

Recall that the degree of a subvariety of a projective space is defined to be the number of its intersection points with a generic linear subspace of complementary dimension. The degree $d$ of $F$, considered as a subvariety of a projective space thanks to Pl\"ucker embedding, is computed by the following formulas:
\begin{equation}
d = \int_G[F]\cdot \sigma_1^{\delta(n,\underline{d},k)},
\end{equation}
where $\int_G\alpha$ denotes the degree of the $0$-dimensional cycle class $\alpha$ on $G$ defined in \cite[Definition 1.4]{F}. If the expected dimension $\delta(n,\underline{d},k) = 1$, then the degree $d$ and genus $g$ of $F$ is computed as follows:
$$d = \int_G[F]\cdot \sigma_1$$
and
$$g  =  1 - \chi(\mathcal O_F) =  1 - \frac{1}{2}\int_Fc_1(T_F),$$
where $T_F$ is the tangent bundle of $F$. In order to determine $c_1(T_F)$, we consider the normal bundle sequence
$$0 \longrightarrow T_F \longrightarrow T_G|_F \longrightarrow N_{F/G} \longrightarrow 0,$$
where $N_{F/G}$ is the normal bundle of $F$ in $G$. By Lemma \ref{lem3}, $F$ is the zero locus of a section of $\mathcal F$, so $N_{F/G}$ is isomorphic to $\mathcal F|_F$. By Lemma \ref{lem1} and Lemma \ref{lem2}, we have
\begin{eqnarray*}
\int_F c_1(T_F) & = & \int_F (c_1(T_G|_F) - c_1(\mathcal F|_F)) \\
		        & = & \int_G (c_1(T_G) - c_1(\mathcal F))\cdot [F]\\ 
		        & = & \int_G\left((n+1)\sigma_1 - \sum_{i=1}^rc_1(\Sym^{d_i}S^\vee)\right)\cdot[F]\\
		 & = & \int_G \left((n+1) - \sum_{i=1}^r \binom{d_i+k}{k+1}\right)\sigma_1\cdot[F]\\
		 & = & \left(n+1 - \sum_{i=1}^r \binom{d_i+k}{k+1}\right)\int_G\sigma_1\cdot[F]\\
		 & = & \left(n+1 - \sum_{i=1}^r \binom{d_i+k}{k+1}\right)d.
\end{eqnarray*}
In summary, the genus-degree formula is obtained as desired.
\end{proof}

\begin{example}\cite[Section 2]{T}\label{exam1}
Let $X\subset\mathbb P^4$ be a general quartic threefold. In this case, $\delta(4,(4),1) = 1$ and $F_1(X)$ is a smooth projective curve of degree $d = 320$ and genus $g = 801$ satisfying
$$g = 1 + \frac{\displaystyle\binom{5}{2}-5}{2}d = 1 + \frac{5}{2}d.$$
More generally, if $X \subset \mathbb P^n \ (n \geq 4)$ be a general hypersurface of degree $2n-4$, then $F_1(X) \subset \mathbb P^{\binom{n+1}{2}-1}$ is a smooth projective curve of degree $d$ and genus $g$ satisfying the following formula:
$$g = 1 + \frac{\displaystyle\binom{2n-3}{2}-n-1}{2}d.$$
\end{example} 

\begin{example}\cite[Theorem 2.2 (i)]{M}\label{exam2}
Let $X \subset \mathbb P^5$ be a general complete intersection of type $(2,3)$. In this case, $\delta(5,(2,3),1) = 1$ and $F_1(X)$ is a smooth projective curve of degree $d = 180$ and genus $g=271$ satisfying
$$g = 1 + \frac{\displaystyle\binom{3}{2}+\binom{4}{2}-6}{2}d = 1 + \frac{3}{2}d.$$
More generally, if $X \subset \mathbb P^n \ (n \geq 5)$ be a general complete intersection of type $(n-3,n-2)$, then $F_1(X)$ is a smooth projective curve of degree $d$ and genus $g$ satisfying the following formula:
$$g = 1 + \frac{\displaystyle\binom{n-2}{2}+\binom{n-1}{2}-n-1}{2}d.$$
\end{example}

\begin{example}\cite[Theorem 2.2 (ii)]{M}\label{exam3}
Let $X \subset \mathbb P^6$ be a general complete intersection of type $(2,2,2)$. Then $\delta(6,(2,2,2),1) = 1$ and $F_1(X)$ is a smooth projective curve of degree $d = 128$ and genus $g=129$ satisfying
$$g = 1 + \frac{\displaystyle\binom{3}{2}+\binom{3}{2} + \binom{3}{2} - 7}{2}d = 1 + d.$$
More generally, if $X \subset \mathbb P^n \ (n \geq 6)$ be a general complete intersection of type $(2,n-4,n-4)$, then $F_1(X)$ is a smooth projective curve of degree $d$ and genus $g$ satisfying the following formula:
$$g = 1 + \frac{\displaystyle2\binom{n-3}{2}-n+2}{2}d.$$
\end{example}

\section{Proof of Theorem \ref{degre}} 

By the Gauss-Bonnet formula (see, for example, \cite[Subsection 3.5.3]{Man}), the class of $F_k(X)$ is the top Chern class of the vector bundle $\mathcal F$. If $\delta(n,\underline{d},k) \geq 0$, then the degree $d$ of $F_k(X)$ can be expressed as follows:
\begin{equation}\label{deg}
d 	= \int_{G(k+1,n+1)} \prod_{i=1}^rc_{\ctop}(\Sym^{d_i}S^\vee)\cdot c_1(S^\vee)^{\delta(n,\underline{d},k)},
\end{equation}
where $c_{\ctop}(E)$ is the top Chern class of the vector bundle $E$. By the splitting principle (\cite[Remark 3.2.3 and Example 3.2.6]{F}), the characteristic class 
$$\prod_{i=1}^rc_{\ctop}(\Sym^{d_i}S^\vee)\cdot c_1(S^\vee)^{\delta(n,\underline{d},k)}$$ 
is represented by the symmetric polynomial
$$(-1)^{(k+1)(n-k)}\prod_{i=1}^r\prod_{a_0+\cdots+a_k=d_i,a_i\in\mathbb N}(a_0x_0+\cdots+a_kx_k)(x_0+\cdots+x_k)^{\delta(n,\underline{d},k)}.$$
Note that $x_0,\ldots,x_n$ are the Chern roots of the tautological subnumdle $S$ on the Grassmannian $G(k+1,n+1)$.
By \cite[Theorem 3]{H1}, Theorem \ref{degre} follows.

As a corollary, we have the following result.

\begin{corollary}
Suppose that $\delta(n,\underline{d},k) = 0$. Then the number of linear subspaces of dimension $k$ contained in a generic complete intersection of type $\underline{d} = (d_1,\ldots,d_r)$ in $\mathbb P^n$ is equal to $\frac{c(n,\underline{d},k)}{(k+1)!}$, where $c(n,\underline{d},k)$ is the coefficient of the monomial $x_0^n \cdots x_k^n$ in the polynomial
$$\prod_{i=1}^r\prod_{a_0+\cdots+a_k=d_i,a_i\in\mathbb N}(a_0x_0+\cdots+a_kx_k)\prod_{i\neq j}(x_i-x_j).$$
\end{corollary}

\subsection*{Acknowledgements}
This work is finished during the author's postdoctoral fellowship at the National Center for Theoretical Sciences (NCTS), Taipei, Taiwan. He would like to thank the NCTS, especially Jungkai Chen, for the financial support and hospitality.


\begin{thebibliography}{99}

\bibitem{AK}
Altman, Allen B. and Kleiman, Steven L.. Foundations of the theory of Fano schemes. Compositio Math. {\bf 34} (1977), 3--47.

\bibitem{BV}
Barth, W. and Van de Ven, A.. Fano varities of lines on hypersurfaces. Arch. Math. (Basel) {\bf 31} (1978), 96--104.

\bibitem{BM}
Bloch, S. and Murre, J. P.. On the Chow groups of certain types of Fano threefolds. Compositio Math. {\bf 39} (1979), 47--105. 

\bibitem{B}
Borcea, Ciprian. Deforming varieties of $k$-planes of projective complete intersections. Pacific J. Math. {\bf 143} (1990), 25--36.

\bibitem{CG}
Clemens, C. H. and Griffiths, P. A.. The intermediate Jacobian of the cubic threefold. Ann. of Math. (2) {\bf 95} (1972), 281--356.

\bibitem{DM}
Debarre, Olivier and Manivel, Laurent. Sur la vari\'{e}t\'{e} des espaces lin\'{e}aires contenus dans une intersection compl\`{e}te. Math. Ann. {\bf 312} (1998), 549--574.

\bibitem{E}
Eisenbud, David. Commutative Algebra with a view toward Algebraic Geometry. Springer-Verlag, 1995.

\bibitem{EH}
Eisenbud, David and Harris, Joe. {\it $3264$ $\&$ All that: A second course in algebraic geometry}; Cambridge University Press, 2016.
 
\bibitem{ELV}
Esnault, H., Levine, M. and Viehweg, E.. Chow groups of projecive varieies of very small degree. Duke Math. J. {\bf 87} (1997), 29--58.

\bibitem{F}
Fulton, William. Intersection theory. Springer-Verlag, 1997.

\bibitem{H}
Hiep, Dang Tuan. On the degree of Fano schemes of linear subspaces on hypersurfaces. Kodai Math. J. {\bf 39} (2016), 110--118.

\bibitem{H1}
Hiep, Dang Tuan. Identities involving (doubly) symmetric polynomials and integrals over Grassmannians. \href{https://arxiv.org/abs/1607.04850}{arXiv:1607.04850}.

\bibitem{J}
Jiang, Zhi. A Noether-Lefschetz theorem for varieties of $r$-planes in complete intersections. Nagoya Math. J. {\bf 206} (2012), 39--66.

\bibitem{L}
Langer, Adrian. Fano schemes of linear spaces on hypersurfaces.  Manuscripta Math. {\bf 93} (1997), 21--28.

\bibitem{Man}
Manivel, Laurent. Symmetric functions, Schubert polynomials and degeneracy loci. American Mathematical Society, 2001.

\bibitem{M}
Markushevich, D. G.. Numerical invariants of families of lines on some Fano varieties. (Russian) Mat. Sb. (N.S.) 116(158) (1981), 265--288. English transl. in Math. USSR-Sb. {\bf 44} (1983), 239--260.

\bibitem{T}
Tennison, B. R.. On the quartic threefold. Proc. London Math. Soc. (3) {\bf 29} (1974), 714--734.
\end{thebibliography}
\end{document}